\documentclass{proc-l}

\usepackage[hidelinks]{hyperref}

\newtheorem{theorem}{Theorem}[section]
\newtheorem{lemma}{Lemma}[section]
\newtheorem{corollary}{Corollary}[section]

\theoremstyle{definition}

\theoremstyle{remark}
\newtheorem{remark}{Remark}[section]

\numberwithin{equation}{section}

\begin{document}

\title[Sums of cusp form coefficients twisted with additive characters]{On sums of Fourier coefficients of cusp forms twisted with additive characters}

\author{Zihao Liu}
\address{Department of Mathematics, University College London}
\curraddr{}
\email{zihao.liu.22@ucl.ac.uk}
\thanks{}


\date{}

\dedicatory{}


\begin{abstract}
    When $a_n$ is the $n$'th coefficient of some holomorphic cusp form, we prove a variety of Omega results for the twisted sum $\sum_{n\le x}a_ne^{2\pi in\alpha}$ and discuss their applications to the Ramanujan $\tau$-function and sums of $a_n$ over arithmetic progressions.
\end{abstract}

\maketitle

    \section{Introduction}\label{sc-intro}
    Let $\Gamma=SL_2(\mathbb Z)$. $f$ is said to be a holomorphic cusp form of weight $k>0$ on $\Gamma$ if it is holomorphic on $\mathbb H\cup\{\infty\}$ such that $f(\infty)=0$ and for any $\gamma=\begin{pmatrix}a&b\\c&d\end{pmatrix}\in\Gamma$ with $c>0$, there exists some number $A(\gamma)$ of modulus unity satisfying
    \begin{equation}
        \label{eqn-modular}
        f\left(a\tau+b\over c\tau+d\right)=A(\gamma)[-i(c\tau+d)]^kf(\tau).
    \end{equation}
    Assume also that $f$ can be expanded at $\infty$ into
    \begin{equation}
        f(\tau)=\sum_{n\ge1}a_ne^{2\pi in\tau}
    \end{equation}
    such that not all the $a_n$'s are zero, so $f(\tau+1)=f(\tau)$, meaning that $A(\gamma)$ is unchanged if $\gamma$ is replaced with $\gamma'=\begin{pmatrix}1&m\\0&1\end{pmatrix}\gamma\begin{pmatrix}1&n\\0&1\end{pmatrix}$ for any integral $m,n$. As a result, knowing either two of $a,b,c,d$ in $\gamma$ is sufficient to determine the value of $A(\gamma)$.

    In this paper, we derive Omega results for a twisted partial sum over $a_n$:
    \begin{equation}
        S(x,\alpha)=\sum_{n\le x}a_ne^{2\pi in\alpha}.
    \end{equation}
    Specifically, we show that
    \begin{theorem}
        \label{th-omega}
        When $\alpha\in\mathbb Q$, there is $S(x,\alpha)=\Omega(x^{\frac k2-\frac14})$.
    \end{theorem}
    If we impose more conditions on $a_n$, we can strengthen this to an $\Omega_\pm$ theorem:
    \begin{theorem}
        \label{th-omegapm}
        For $\gamma=\begin{pmatrix}a&b\\c&d\end{pmatrix}\in\Gamma$, if there exists some $n\in\mathbb N$ such that $\Re[A(\gamma)a_ne^{-2\pi ind/c}]\ne0$, then when $\alpha=a/c$, we have $\Re[S(x,\alpha)]=\Omega_\pm(x^{\frac k2-\frac14})$.
    \end{theorem}
    \autoref{th-omega} and \autoref{th-omegapm} are proved based on the asymptotic behavior of the following generating series
    \begin{equation}
        \label{eqn-Fs}
        F(s,\alpha)=\sum_{n\ge1}a_ne^{2\pi in\alpha}e^{-s\sqrt n}
    \end{equation}
    when $s$ approaches the imaginary axis from the right half plane.
    \begin{theorem}
        \label{th-Fs}
        When $\alpha=a/c$ for some $(a,c)=1$, $\sigma>0$, and $t_n=4\pi\sqrt n/c$, there exists nonzero constants $B^+,B^-$ only depending on $k$ such that as $\sigma\to0^+$, $F(\sigma,\alpha)=O(\sigma)$ and
        \begin{equation*}
            F(\sigma\pm it_n)=
            \begin{cases}
                [1+o(1)]B^\pm A(\gamma)a_ne^{-2\pi ind/c}c^{-k}\sigma^{-k-\frac12}, & a_n\ne 0, \\
                O(1), & a_n=0,
            \end{cases}
        \end{equation*}
        in which $b$ and $d$ are any integers satisfying $\gamma=\begin{pmatrix}a&b\\c&d\end{pmatrix}\in\Gamma$.
    \end{theorem}

    Omega results concerning the Fourier coefficients of modular forms are usually proved by studying the second moments. For instance, the theory of Rankin--Selberg integrals \cite{iwaniec_topics_1997} produces
    \begin{equation}
        \label{eqn-rs}
        \sum_{n\le x}|a_n|^2\asymp x^k,
    \end{equation}
    immediately producing $a_n=\Omega(n^{k-1\over2})$, which is essentially optimal in view of the Ramanujan--Petersson conjecture. Walker \cite{walker_sums_2018} applied this method to study the second moments of $S(x,0)$, thereby proving the $\alpha=0$ special case of \autoref{th-omega}.

    Our method is motivated by G. H. Hardy's investigations of the classical lattice-point problems \cite{hardy_expression_1915}\cite{hardy_dirichlets_1917}, in which it was proved that as $x\to+\infty$
    \begin{equation*}
        \sum_{n\le x}d(n)-[x\log x+(2\gamma-1)\log x]=\Omega_\pm(x^{\frac14}),
    \end{equation*}
    \begin{equation*}
        \sum_{n\le x}r_2(n)-\pi x=\Omega_\pm(x^{\frac14}),
    \end{equation*}
    where $d(n)$ is the divisor function and $r_2(n)$ is the number of ways to express $n$ as a sum of two squares. It should be noted that another means to yield Omega results is by appealing to the general theorems of Chandrasekharan and Narasimhan \cite{chandrasekharan_functional_1962}\cite{chandrasekharan_mean_1964}, but due to the generality of their arguments, more stringent assumptions have to be imposed on $a_n$ hein order to apply their theorems.

    The paper is organized in the following manner: In \autoref{sc-an-series}, we study the asymptotic behavior of $F(s,\alpha)$ and deduce \autoref{th-Fs}. In \autoref{sc-omega}, we applied the properties of $F(s)$ to prove \autoref{th-omega}. In \autoref{sc-omegapm}, we perform a more refined analysis and derive \autoref{sc-omegapm}. Finally, in \autoref{sc-apps}, we apply our results to Ramanujan $\tau$-function and sums of Fourier coefficients over arithmetic progressions.

    \section{The generating series $F(s)$}\label{sc-an-series}
    From now on, $\alpha=a/c$ for some integers $(a,c)=1$. Choose integers $b,d$ such that $ad-bc=1$ so $\gamma=\begin{pmatrix}a&b\\c&d\end{pmatrix}\in\Gamma$. Write
    \begin{equation*}
        \tau'=\frac ac+{iu^{-1}\over c},\quad\tau=-\frac dc+{iu\over c},
    \end{equation*}
    so that \eqref{eqn-modular} is equivalent to
    \begin{equation}
        \label{eqn-modular2}
        \begin{aligned}
        f(\tau')
        &=\sum_{n\ge1}a_ne^{2\pi ina/c}e^{-{2\pi n\over c}\cdot u^{-1}} \\
        &=A(\gamma)u^kf(\tau)=A(\gamma)\sum_{n\ge1}a_ne^{-2\pi ind/c}u^ke^{-{2\pi n\over c}\cdot u}.
        \end{aligned}
    \end{equation}

    When $s>0$, the integral identity
    \begin{equation}
        \label{eqn-esn}
        e^{-s\sqrt n}={s\over\sqrt{\pi}}\left(c\over8\pi\right)^{\frac12}\int_0^{+\infty}u^{-\frac12}e^{-\left(c\over8\pi\right)s^2u}e^{-{2\pi n\over c}\cdot u^{-1}}\mathrm du,
    \end{equation}
    transforms \eqref{eqn-Fs} into
    \begin{equation*}
        F(s,\alpha)={s\over\sqrt{\pi}}\left(c\over8\pi\right)^{\frac12}\int_0^{+\infty}u^{-\frac12}e^{-\left(c\over8\pi\right)s^2u}f\left(\frac ac+{iu^{-1}\over c}\right)\mathrm du.
    \end{equation*}
    Plugging in \eqref{eqn-modular2} gives
    \begin{equation*}
        F(s,\alpha)={A(\gamma)s\over\sqrt{\pi}}\left(c\over8\pi\right)^{\frac12}\sum_{n\ge1}a_ne^{-2\pi ind/c}\int_0^{+\infty}u^{k-\frac12}e^{-\left[\left(c\over8\pi\right)s^2+{2\pi n\over c}\right]u}\mathrm du.
    \end{equation*}
    Computing the remaining integral and simplifying yields the following identity:
    \begin{equation}
        \label{eqn-Fs2}
        F(s,\alpha)=A(\gamma)s\left(8\pi\over c\right)^k{\Gamma(k+\frac12)\over\sqrt{\pi}}\sum_{n\ge1}{a_ne^{-2\pi ind/c}\over(s^2+t_n^2)^{k+\frac12}}.
    \end{equation}
    Due to Hecke's bound $a_n=O(n^{\frac k2})$ and $t_n^2\asymp n$, the right-hand side converges in any compact subset $\Re(s)>0$, so \eqref{eqn-Fs2} is valid throughout the half plane $\Re(s)>0$ by analytic continuation.
    \begin{proof}[Proof of \autoref{th-Fs}]
        When $s\to0$, it follows from \eqref{eqn-Fs2} that $F(s,\alpha)/s$ is bounded, so $F(\sigma,\alpha)=O(\sigma)$ for $\sigma>0$. When $s=\sigma\pm it_n$, notice that
        \begin{equation*}
            s^2+t_n^2=(s+it_n)(s-it_n)\sim\pm 2it_n\sigma=2e^{\pm i\pi/2}\sigma,
        \end{equation*}
        so we have
        \begin{align*}
            F(s,\alpha)
            &=A(\gamma)(\pm it_n)\left(8\pi\over c\right)^k{\Gamma(k+\frac12)\over2^{k+\frac12}\sqrt{\pi}}\cdot a_ne^{-2\pi ind/c}e^{\mp(k+\frac12){\pi i\over2}}\sigma^{-k-\frac12}+o(\sigma^{-k-\frac12}) \\
            &=\underbrace{e^{\pm(\frac12-k){\pi i\over2}}(4\pi)^k{\Gamma(k+\frac12)\over\sqrt{2\pi}}}_{B^\pm}\cdot A(\gamma)a_ne^{-2\pi ind/c}c^{-k}\sigma^{-k-\frac12}+o(\sigma^{-k-\frac12}).
        \end{align*}
    \end{proof}
    \begin{remark}
        G. H. Hardy \cite{hardy_expression_1915} proved a corresponding identity for $r_2(n)$ and obtained a similar asymptotic result as $\Re(s)\to0^+$, but there are no analogs of \eqref{eqn-Fs2} in divisor problems \cite{hardy_dirichlets_1917}.
    \end{remark}
    \renewcommand{\proofname}{Proof}
    \section{Omega results for the magnitude of partial sums}\label{sc-omega}
    By partial summation, we can express $F(s,\alpha)$ as an integral of $S(x,\alpha)$:
    \begin{equation*}
        F(s,\alpha)=\int_0^{+\infty}e^{-sx^{\frac12}}\mathrm dS(x,\alpha)=\int_0^{+\infty}e^{-sy}\mathrm dS(y)=s\int_0^{+\infty}S(y)e^{-sy}\mathrm dy.
    \end{equation*}
    If $|S(x,\alpha)|<\varepsilon x^\theta$ when $x>x_0$, then whenever $0<\Re(s)=\sigma\le1$ and $\Im(s)$ is fixed, there is
    \begin{equation*}
        |F(s,\alpha)|<\varepsilon|s|\int_{x_0}^{+\infty}y^{2\theta}e^{-\sigma y}\mathrm dy+O_{x_0}(1)=\varepsilon|s|\Gamma(2\theta+1)\sigma^{-2\theta-1}+O_{x_0}(1),
    \end{equation*}
    which indicates that
    \begin{lemma}
        \label{lm-o}
        If $S(x,\alpha)=o(x^\theta)$ as $x\to+\infty$, then $F(\sigma+it,\alpha)=o(\sigma^{-2\theta-1})$ when $t\in\mathbb R\setminus\{0\}$ is fixed and $\sigma\to0^+$.
    \end{lemma}
    \begin{proof}[Proof of \autoref{th-omega}]
        By our assumption on $f$ in \autoref{sc-intro}, we can choose $n\in\mathbb N$ such that $a_n\ne0$, so setting $\theta=\frac k2-\frac14$ and $t=t_n$ in \autoref{lm-o} contradicts \autoref{th-Fs}. Thus, $S(x,\alpha)=\Omega(x^{\frac k2-\frac14})$ (i.e. the negation of $o(x^{\frac k2-\frac14})$).
    \end{proof}
    \section{Oscillations of real components of partial sums}\label{sc-omegapm}
    In this section, we only give a detailed proof of $\Re[S(x,\alpha)]=\Omega_+(x^{\frac k2-\frac14})$, the $\Omega_-$ theorem can be proved by formulating a symmetrical reasoning.

    If $\Re[S(x,\alpha)]<\varepsilon x^\theta$ when $x>x_0$ and define
    \begin{equation}
        \label{eqn-Gs}
        G(s)=s\int_0^{+\infty}\{\varepsilon y^{2\theta}-\Re[S(y^2)]\}e^{-sy}\mathrm dy,
    \end{equation}
    then whenever $0<\Re(s)=\sigma\le1$ and $\Im(s)$ is fixed, there is
    \begin{align*}
        |G(s)|
        &<|s|\int_{x_0}^{+\infty}[\varepsilon y^{2\theta}-\Re[S(y^2)]]e^{-\sigma y}\mathrm dy+O_{x_0}(1) \\
        &=\varepsilon|s|\int_0^{+\infty}y^{2\theta}e^{-\sigma y}\mathrm dy-|s|\Re\left[\int_0^{+\infty}S(y^2)e^{-\sigma y}\mathrm dy\right]+O_{x_0}(1). \\
        &=\varepsilon|s|\Gamma(2\theta+1)\sigma^{-2\theta-1}-|s|\sigma^{-1}\Re[F(\sigma,\alpha)]+O_{x_0}(1).
    \end{align*}
    By \autoref{th-Fs}, $F(\sigma,\alpha)=O(\sigma)$, so the second term is bounded as $\sigma\to0^+$, which indicates that
    \begin{lemma}
        \label{lm-op}
        If $\limsup_{x\to+\infty}\Re[S(x,\alpha)]x^{-\theta}\le 0$, then $G(\sigma+it)=o(\sigma^{-2\theta-1})$ whenever $t\in\mathbb R$ is fixed and $\sigma\to0^+$.
    \end{lemma}
    \begin{proof}[Proof of \autoref{th-omegapm}]
        Let $F_1(s)=\frac12[F(s,\alpha)+\overline{F(\overline s,\alpha)}]$, so it follows from \eqref{eqn-Fs} and partial summation that when $\alpha=a/c$ for some $(a,c)=1$, there is
        \begin{equation}
            \label{eqn-F1s}
            F_1(s)=\sum_{n\ge1}\Re(a_ne^{2\pi ina/c})e^{-s\sqrt{n}}=s\int_0^{+\infty}\Re[S(y^2)]e^{-sy}\mathrm dy.
        \end{equation}
        When $a_n\ne0$ for some $n$, $s=\sigma+it_n$, and $\sigma\to0^+$, it follows from \autoref{th-Fs} that \eqref{eqn-F1s} becomes
        \begin{equation*}
            F_1(s)=\frac12[B^+A(\gamma)a_ne^{-2\pi ind/c}+\overline{B^-A(\gamma)a_ne^{-2\pi ind/c}}]c^{-k}\sigma^{-k-\frac12}+o(\sigma^{-\frac k2-\frac12})
        \end{equation*}
        According to the proof of \autoref{th-Fs}, $\overline{B^-}=B^+$, so
        \begin{equation}
            \label{eqn-F1s-asymp}
            F_1(s)=B^+\Re[A(\gamma)a_ne^{-2\pi ind/c}]c^{-k}\sigma^{-\frac k2-\frac12}+o(\sigma^{-k-\frac12}).
        \end{equation}
        Plugging \eqref{eqn-F1s} into \eqref{eqn-Gs}, we have
        \begin{align*}
            G(s)
            &=\varepsilon\Gamma(2\theta+1)s^{-2\theta}-F_1(s) \\
            &=B^+\Re[A(\gamma)a_ne^{-2\pi ind/c}]c^{-k}\sigma^{-\frac k2-\frac12}+o(\sigma^{-k-\frac12}).
        \end{align*}
        Due to our assumption, we can choose $n\in\mathbb N$ such that $\Re[A(\gamma)a_ne^{-2\pi ind/c}]\ne0$. Plugging $\theta=\frac k2-\frac14$ into \autoref{lm-op} leads to a contradiction, so we must have $\Re[S(x,\alpha)]=\Omega_+(x^{\frac k2-\frac14})$.
    \end{proof}
    \section{Examples and Applications}\label{sc-apps}
    \subsection{Ramanujan $\tau$-function}
    If $f(\tau)=\Delta(\tau)=\sum_{n\ge1}\tau(n)e^{2\pi in\tau}$, then $k=12$ and $A(\gamma)=1$, so it follows from \autoref{th-omega} that
    \begin{corollary}
        For any integer $(a,c)=1$, we have
        \begin{equation*}
            \sum_{n\le x}\tau(n)e^{2\pi ina/c}=\Omega(x^{6-\frac14}).
        \end{equation*}
    \end{corollary}
    As for the $\Omega_\pm$ result, observe that
    \begin{equation*}
        \Re[A(\gamma)\tau(n)e^{-2\pi ind/c}]=\tau(n)\cos\left(2\pi nd\over c\right),
    \end{equation*}
    so according to \autoref{th-omegapm}, we have
    \begin{corollary}
        For any integer $(a,c)=1$, if $c$ is odd or there exists some $n\in\mathbb N$ such that $\tau(n)\ne0$ and $na\not\equiv\pm\frac c2\pmod c$, then
        \begin{equation*}
            \sum_{n\le x}\tau(n)\cos\left(2\pi na\over c\right)=\Omega_\pm(x^{6-\frac14}).
        \end{equation*}
    \end{corollary}
    \subsection{Sums over arithmetic progressions}
    Define
    \begin{equation*}
        S(x;q,h)=\sum_{\substack{n\le x\\n\equiv h\pmod q}}a_n.
    \end{equation*}
    Then from the properties of additive characters, we have
    \begin{corollary}
        \label{co-ap}
        For every $q\in\mathbb N$, there exists some $1\le h\le q$ such that $(h,q)=1$ and $S(x;q,h)=\Omega(x^{\frac k2-\frac14})$.
    \end{corollary}
    \begin{proof}
        For all integers $a,q\ge1$, there is
        \begin{align*}
            S(x,\alpha)
            &=\sum_{n\le x}a_ne^{2\pi ina/q}
            =\sum_{1\le r\le q}e^{2\pi ir/q}\sum_{\substack{n\le x\\rn\equiv a\pmod q}}a_n \\
            &=\sum_{1\le r\le q}e^{2\pi ir/q}S(x;q,h_r)\quad(rh_r\equiv1\pmod q).
        \end{align*}
        If $S(x;q,h)=o(x^{\frac k2-\frac14})$ for all $1\le h\le q$, then $S(x,\alpha)=o(x^{\frac k2-\frac14})$, which will contradict \autoref{th-omega}.
    \end{proof}
    By partial summation, \autoref{co-ap} can also be phrased using normalized Fourier coefficients:
    \begin{corollary}
        \label{co-ap2}
        Let $\hat a_n=a_n/n^{k-1\over2}$ be the normalized $n$'th Fourier coefficient of $f$. Then for every $q\in\mathbb N$, there exists some $1\le h\le q$ such that $(h,q)=1$ and
        \begin{equation*}
            \hat S(x;q,h)=\sum_{\substack{n\le x\\n\equiv h\pmod q}}\hat a_n=\Omega(x^{\frac14}).
        \end{equation*}
    \end{corollary}
    In the $O$-direction, L\"u \cite{lu_average_2009} proved that for all $q\ge1$, there is
    \begin{equation*}
        \sum_{1\le h\le q}|\hat S(x;q,h)|^2\ll x\log x,
    \end{equation*}
    where $\ll$ only depend on $f$. Combining this with \autoref{co-ap2}, we obtain an upper bound for $\tilde S(x;q,h)$ and $S(x;q,h)$ valid ``almost everywhere.''
    \begin{corollary}
        If $x^{\frac12+\varepsilon}\le q\le x$, then
        \begin{equation*}
            \tilde S(x;q,h)=O(x^{\frac14}),\quad S(x;q,h)=O(x^{\frac k2-\frac14})
        \end{equation*}
        for almost all $1\le h\le q$.
    \end{corollary}
    This demonstrates that the exponents of our Omega bounds in \autoref{co-ap} and \autoref{co-ap2} are optimal.
    \subsection{Further discussions on $S(x;q,h)$}
    Although \autoref{co-ap} demonstrates that we are unable to reduce the exponent $\frac k2-\frac14$ in the bound of $S(x;q,h)$, it is not an effective result that allows us to determine which $h$ fulfills the Omega bound.

    From the orthogonality of additive characters, we have
    \begin{equation}
        \label{eqn-sxqh}
        \begin{aligned}
            S(x;q,h)
            &=\frac1q\sum_{1\le m\le q}e^{-2\pi ihm/q}S\left(x,\frac mq\right) \\
            &=\frac1q\sum_{c|q}\sum_{\substack{1\le a\le c\\(a,c)=1}}e^{-2\pi iha/c}S\left(x,\frac ac\right).
        \end{aligned}
    \end{equation}
    Similar to how we prove \autoref{th-omega}, define
    \begin{equation*}
        F(s;q,h)=\sum_{\substack{n\ge1\\n\equiv h\pmod q}}a_ne^{-s\sqrt n}=s\int_0^{+\infty}S(y^2;q,h)e^{-sy}\mathrm dy,
    \end{equation*}
    so it follows from \eqref{eqn-sxqh} that
    \begin{equation*}
        F(s;q,h)=\frac1q\sum_{c|q}\sum_{\substack{1\le a\le c\\(a,c)=1}}e^{-2\pi iha/c}F\left(x,\frac ac\right).
    \end{equation*}
    Let $b$ and $d$ be chosen such that $ad-bc=1$ and $\gamma=\begin{pmatrix}a&b\\c&d\end{pmatrix}\in\Gamma$. Then it follows from \autoref{th-Fs} that when $s=\sigma+it_n$ such that $a_n\ne0$, there is
    \begin{equation*}
        F(s;q,h)={B^+a_n\over q\sigma^{k+\frac12}}\sum_{c|q}{K_f(-h,-n;c)\over c^k}+o(\sigma^{-\frac k2-\frac12}),
    \end{equation*}
    where $K_f(h,n;c)$ is the Kloosterman sum associated with $f$:
    \begin{equation*}
        K_f(m,n;c)=\sum_{\substack{1\le a\le c\\(a,c)=1}}A(\gamma)e^{2\pi ima/c+2\pi ind/c}.
    \end{equation*}
    By adapting the arguments in \autoref{sc-omega}, we deduce
    \begin{corollary}
        For every $q\in\mathbb N$, if there exist $1\le h\le q$ and $n\in\mathbb N$ such that
        \begin{equation*}
            a_n\sum_{c|q}{K_f(-h,-n;c)\over c^k}\ne0,
        \end{equation*}
        then as $x\to+\infty$, we have
        \begin{equation*}
            S(x;q,h)=\Omega(x^{\frac k2-\frac14}),\quad\tilde S(x;q,h)=\Omega(x^{\frac14}).
        \end{equation*}
    \end{corollary}
    Although this result requires more conditions on $a_n$, it is considerably more effective than \autoref{co-ap} and \autoref{co-ap2}.
    \bibliographystyle{amsplain}
    \bibliography{refs}

\providecommand{\bysame}{\leavevmode\hbox to3em{\hrulefill}\thinspace}
\providecommand{\MR}{\relax\ifhmode\unskip\space\fi MR }
\providecommand{\MRhref}[2]{%
  \href{http://www.ams.org/mathscinet-getitem?mr=#1}{#2}
}
\providecommand{\href}[2]{#2}
\begin{thebibliography}{1}

\bibitem{chandrasekharan_functional_1962}
K.~Chandrasekharan and Raghavan Narasimhan, \emph{Functional {Equations} {With} {Multiple} {Gamma} {Factors} and the {Average} {Order} of {Arithmetical} {Functions}}, The Annals of Mathematics \textbf{76} (1962), no.~1, 93.

\bibitem{chandrasekharan_mean_1964}
\bysame, \emph{On the mean value of the error term for a class of arithmetical functions}, Acta Mathematica \textbf{112} (1964), no.~0, 41--67 (en).

\bibitem{hardy_expression_1915}
G.~H. Hardy, \emph{On the expression of a number as the sum of two squares}, Quarterly Journal of Mathematics \textbf{46} (1915), 263--283.

\bibitem{hardy_dirichlets_1917}
\bysame, \emph{On {Dirichlet}'s {Divisor} {Problem}}, Proceedings of the London Mathematical Society \textbf{15} (1917), no.~1, 1--25 (en).

\bibitem{iwaniec_topics_1997}
Henryk Iwaniec, \emph{Topics in classical automorphic forms}, Graduate studies in mathematics, no. v. 17, American Mathematical Society, Providence, R.I, 1997.

\bibitem{lu_average_2009}
Guangshi Lü, \emph{The average value of {Fourier} coefficients of cusp forms in arithmetic progressions}, Journal of Number Theory \textbf{129} (2009), no.~2, 488--494 (en).

\bibitem{walker_sums_2018}
Alexander~Weston Walker, \emph{Sums of {Fourier} {Coefficients} of {Modular} {Forms} and the {Gauss} {Circle} {Problem}}, {PhD} {Thesis}, Brown University, 2018.

\end{thebibliography}
\end{document}